\documentclass[a4paper]{amsart}
\usepackage[english]{babel}
\usepackage{amsmath,amssymb,amsthm}
\usepackage{graphicx}
\usepackage{color}
\usepackage[active]{srcltx}

\theoremstyle{plain}
\newtheorem{theorem}{Theorem}[section]

\newtheorem{proposition}[theorem]{Proposition}
\newtheorem{lemma}[theorem]{Lemma}

\theoremstyle{definition}

\newtheorem{definition}[theorem]{Definition}

 \DeclareMathOperator{\dist}{dist\,}
 
 \DeclareMathOperator{\supp}{supp}
\DeclareMathOperator{\osc}{Osc}

\newcommand{\C}{\mathbb{C}}

\newcommand{\N}{\mathbb{N}}

\newcommand{\eps}{\varepsilon}

\renewcommand{\leq}{\leqslant}
\renewcommand{\geq}{\geqslant}
\renewcommand{\le}{\leqslant}
\renewcommand{\ge}{\geqslant}

\begin{document}

\title[Bishop-Phelps-Bollob\'{a}s property in spaces of continuous
functions]{The Bishop-Phelps-Bollob\'{a}s property for
operators between spaces of continuous functions}

\author[Acosta]{Mar\'ia D. Acosta}
\author[Becerra]{Julio Becerra-Guerrero}
\author[Choi]{Yun Sung Choi}
\author[Ciesielski]{Maciej Ciesielski}
\author[Kim]{Sun Kwang Kim}
\author[Lee]{Han Ju Lee}
\author[Louren\c{c}o]{Mary Lilian Louren\c{c}o}
\author[Mart\'{\i}n]{Miguel Mart\'{\i}n}

\address[Acosta, Becerra and Mart\'{\i}n]{Universidad de Granada, Facultad de Ciencias.
Departamento de An\'{a}lisis Matem\'{a}tico, 18071 Granada (Spain)}
\email{\texttt{dacosta@ugr.es}, \texttt{juliobg@ugr.es}, \texttt{mmartins@ugr.es} }

\address[Choi]{Department of Mathematics, POSTECH, Pohang (790-784), Republic of Korea}
\email{\texttt{mathchoi@postech.ac.kr}}

\address[Ciesielski]{Pozna\'{n} University of Tech\-no\-lo\-gy, Piotrowo 3A, 60-965
Pozna\'{n}, Poland}
\email{\texttt{maciej.ciesielski@put.poznan.pl}}

\address[Kim]{School of Mathematics, Korea Institute for Advanced Study (KIAS),
 85 Hoegiro, Dongdaemun-gu, Seoul 130-722, Republic of Korea}
\email{\texttt{lineksk@gmail.com}}

\address[Lee]{Department of Mathematics Education,
Dongguk University - Seoul, 100-715 Seoul, Republic of Korea}
\email{\texttt{hanjulee@dongguk.edu}}

\address[Louren\c{c}o]{Universidade de Sao Paulo, Instituto de Matem\'atica e
Estat\'\i stica, CP 66281 CEP:05311-970  S\~{a}o Paulo, Brazil}
\email{\texttt{mllouren@ime.usp.br}}


\thanks{First and eighth authors partially supported by Spanish MINECO and FEDER
project no.\ MTM2012-31755, Junta de Andaluc\'{\i}a and FEDER grants FQM-185 and
P09-FQM-4911. Second author supported by  MTM2011-23843 and Junta de Andaluc\'ia
grants FQM 0199 and FQM 1215. Third author partially supported by Basic Science
Research Program through the National Research Foundation of Korea (NRF) funded by
the Ministry of Education, Science and Technology (No. 2010-0008543), and also by
Priority Research Centers Program through the National Research Foundation of Korea
(NRF) funded by the Ministry of Education, Science and Technology (MEST) (No.
2012047640). Sixth  author partially supported by Basic Science Research Program
through the National Research Foundation of Korea(NRF) funded by the Ministry of
Education, Science and Technology (2012R1A1A1006869). Seventh  author partially
supported by FAPESP (Project no  2012/01015-9).}

\large

\begin{abstract}
We show that the space of bounded  linear operators between  spaces of continuous
functions on compact Hausdorff topological spaces has the
Bishop-Phelps-Bollob\'{a}s property. A similar result is also proved for the class
of compact operators from the space of continuous functions vanishing at infinity
on a locally compact and Hausdorff topological space into a uniformly convex space,
and for the class of compact operators from a Banach space into a predual of an
$L_1$-space.
\end{abstract}

\date{July 17th, 2013}

\maketitle

\section{Introduction}
E.~Bishop and R.~Phelps proved in 1961 \cite{BP} that every
(continuous linear) functional $x^*$ on a Banach space $X$ can be approximated by a
norm attaining functional $y^*$. This result is called  the Bishop-Phelps Theorem.
Shortly thereafter, B.~Bollob\'{a}s \cite{Bol} showed that this approximation can
be done in such a way that, moreover, the point at which $x^*$ almost attains its
norm is close in norm to a point at which $y^*$ attains its norm. This is a
quantitative version of the Bishop-Phelps Theorem, known as the
Bishop-Phelps-Bollob\'{a}s Theorem.

For a real or complex Banach space $X$, we denote by $S_X$, $B_X$ and $X^*$ the
unit sphere, the closed unit ball and the dual space of $X$, respectively.

\begin{theorem}[{Bishop-Phelps-Bollob\'{a}s Theorem, \cite[Theorem~1]{Bol}}]
Let  $X$ be a Banach space and $0 < \varepsilon <1/2$. Given  $x \in B_X$ and $x^* \in
S_{X^*}$  with $\vert 1 - x^*(x ) \vert < \dfrac{\varepsilon^2}{2}$, there are elements
$y\in S_X $ and $ y^* \in S_{X^*}$ such that $y^* (y)=1$,  $ \Vert y-x \Vert <
\varepsilon + \varepsilon^2$  and $ \Vert y^* -x^* \Vert < \varepsilon $.
\end{theorem}

We refer the reader to the recent paper \cite{CKMMR} for a more accurate version of the above
theorem.

In 2008, M.D. Acosta, R.M. Aron, D. Garc\'{\i}a and M. Maestre introduced the
so-called Bishop-Phelps-Bollob\'{a}s property for operators \cite[Definition
1.1]{AAGM}. For our purposes, it will be useful to recall  an appropriate version
of this property for classes of operators defined  in \cite[Definition 1.3]{ABGKM}.
Given two Banach spaces $X$ and $Y$, $\mathcal{L}(X,Y)$ denotes the space of all
(bounded and linear) operators from $X$ into $Y$. The subspace of
$\mathcal{L}(X,Y)$ of {\it finite-rank operators} $\mathcal{F}(X,Y) $;
$\mathcal{K}(X,Y)$ will denote the subspace of all {\it compact operators}.

\begin{definition}
\label{def-BPBP-operadores}
Let $X$ and $Y$ be Banach spaces and $\mathcal{M}$ a linear subspace of
$\mathcal{L}(X,Y)$. We say that $\mathcal{M}$ satisfies the {\em
Bishop-Phelps-Bollob\'{a}s property} if given $\varepsilon > 0$, there is $ \eta
(\varepsilon)  > 0$   such that whenever   $T \in S_{\mathcal{M}}$ and $x_0 \in
S_{X}$  satisfy that $\Vert  T x_0 \Vert > 1- \eta (\varepsilon )$, then there
exist  a point   $u_0 \in S_{X}$ and an  operator   $S \in S_{\mathcal{M}}$
satisfying the following conditions:
$$
\Vert  Su_0 \Vert  =1, \quad  \Vert u_0- x_0 \Vert < \varepsilon, \quad \text{and} \quad \Vert
S-T \Vert < \varepsilon.
$$
In  case that $\mathcal{M}=\mathcal{L}(X,Y)$ satisfies the previous property it is
said that the pair $(X,Y)$ has the {\it  Bishop-Phelps-Bollob\'{a}s property for
operators} (shortly
\emph{BPBp for operators}).
\end{definition}

Observe that the BPBp of a pair $(X,Y)$ means that one is able to approximate any pair of an
operator and a point at which the operator almost attains its  norm by a new pair of a
norm-attaining operator and a point at which this  new operator attains its norm. In particular,
if a pair $(X,Y)$ has the BPBp, the set of norm-attaining operators is dense in
$\mathcal{L}(X,Y)$. The reverse result is far from being true: there are Banach spaces $Y$ such
that the pair $(\ell_1^2,Y)$ does not have the BPBp (see \cite{AAGM}).

In \cite{AAGM} the authors provided the first version of the
Bishop-Phelps-Bollob\'{a}s Theorem for operators.  Amongst them, a sufficient
condition on a Banach space $Y$ to get that for every Banach space $X$, the pair
$(X,Y)$ has the BPBp for operators, which is satisfied, for instance, by $Y=c_0$ or
$Y=\ell_\infty$. A characterization of the Banach spaces $Y$ such that the pair
$(\ell_1,Y)$ has the BPBp for operators is also given. There are also positive
results for operators from $L_1 (\mu) $ into $L_\infty(\nu)$ \cite{ACGM,CKLM}, for
operators from $L_1(\mu)$ into $L_1(\nu)$ \cite{CKLM}, for certain ideals of
operators from $L_1 (\mu) $ into another Banach space \cite{CK,ABGKM}, for
operators from an Asplund space into $C_0(L)$ or into a uniform algebra
\cite{ACK,CGK}, and for operators from a uniformly convex space into an arbitrary
Banach space \cite{ABGM, KL}. For some more recent results, see also \cite{ACKLM}.
Let us also point out that the set of norm attaining operators from $L_1[0,1]$ into
$C[0,1]$ is not dense in $\mathcal{L}(L_1[0,1], C[0,1])$ \cite{SchaP}.

Our aim in this paper is to provide classes of Banach spaces satisfying a version
of the Bishop-Phelps-Bollob\'{a}s Theorem for operators. The first result, which is
the content of section~\ref{sec:C(K)-C(S)}, states that given arbitrary compact
Hausdorff topological spaces $K$ and $S$, the pair $(C(K), C(S))$ satisfies the
BPBp for operators in the real case. This result extends the one by J.~Johnson and
J.~Wolfe \cite{JoWo} that the set of norm attaining operators from $C(K)$ into
$C(S)$ is dense in $\mathcal{L}(C(K), C(S))$. In
section~\ref{sec:c(K)toUniformlyconvex}, we prove that the space
$\mathcal{K}(C_0(L),Y)$ satisfies the Bishop-Phelps-Bollob\'{a}s property whenever
$L$ is a locally compact Hausdorff topological space and $Y$ is uniformly convex in
both the real and the complex case. Let us remark that it was also proved in
\cite{JoWo} that the set of norm-attaining weakly compact operators from $C(K)$
into $Y$ is dense in the space of all weakly compact operators. But, as commented
above, there are Banach spaces $Y$ such that the pair $(\ell_\infty^2,Y)$ does not
satisfy the BPBp for operators (in the real case, $\ell_\infty^2\equiv \ell_1^2$),
so some assumption on $Y$ is needed to get the Bishop-Phelps-Bollob\'{a}s property.
Finally, we devote section~\ref{sec:intoPredual_of_L_1} to show that the space
$\mathcal{K}(X,Y)$ has the Bishop-Phelps-Bollob\'{a}s property when $X$ is an
arbitrary Banach space and $Y$ is a predual of an $L_1$-space in both the real and
the complex case. This extends the result of \cite{JoWo} that the set of
norm-attaining finite-rank operators from an arbitrary Banach space into a predual
of an $L_1$-space is dense in the space of compact operators. In particular, for
$Y=C_0(L)$ for some locally compact Hausdorff topological space $L$, the result is
a consequence of the already cited paper \cite{ACK}.

\section{Operators between spaces of continuous functions}
\label{sec:C(K)-C(S)}
Throughout this section, $K$ and $S$ are compact Hausdorff
topological spaces.
Here $C(K)$ is  the space of real valued
continuous functions on $K$. $M(K)$ denotes the space of regular
Borel finite measures on $K$, which identifies with the dual of $C(K)$ by the Riesz
representation theorem. For $s\in S$, we write $\delta_s$ to denote the point
measure concentrated at $s$.

\begin{lemma}[\mbox{\cite[Theorem 1, p.~490]{DuSch}}]
\label{lem-isometric-isomorphism}
Let $X$ be a Banach space and let $S$ be a compact Hausdorff topological space. Given an operator
$A:X\longrightarrow{C(S)}$, define $\mu:S\longrightarrow{X^*}$ by $\mu(s)=A^*(\delta_s)$ for
every $s\in S$. Then the relationship
$$
[Ax](s)= \mu(s)(x), \ \ \  \forall x\in X,\, s \in S
$$
defines an isometric isomorphism between $\mathcal{L}(X,C(S))$ and the space of
$w^*$-continuous functions from $S$ to $X^*$, endowed  with the supremum norm,
i.e.\ $\Vert{\mu}\Vert=\sup\{\Vert{\mu(s)}\Vert:s\in{S}\}$.
Compact operators correspond to norm continuous
functions.
\end{lemma}

\begin{lemma}[\mbox{\cite[Lemma 2.2]{JoWo}}]
\label{lem-w*lower-semicontinuous}
Let $\mu:S\longrightarrow{M(K)}$ be $w^*$-continuous. Let $\varepsilon>0$, $s_0\in{S}$ and an
open subset $V$ of $K$ be given. Then there exists an open neighborhood $U$ of $s_0$ such that if
$s\in{U}$, then $|\mu(s)|(V)\geq|\mu(s_0)|(V)-\varepsilon$.
\end{lemma}

The next result is a version of \cite[Lemma 2.3]{JoWo} in which the main difference
is that we start with an operator and a function  in the unit sphere of $C(K)$
where the operator almost attains its norm and construct a new
operator and a new function, both close to the previous elements and satisfying
additional restrictions.  Condition iii) is the new ingredient that will be useful
to our purpose.

\begin{lemma}
\label{funct-attains-close}
Let $\mu: S \longrightarrow M(K)$ be a $w^*$-continuous function  satisfying $\Vert \mu \Vert =1$
and  $0 < \delta < 1$. Suppose that $s_0  \in S$ and $f_0 \in S _{C(K)}$ satisfy $\int_K f_0 \ d
\mu (s_0) > 1 - \frac{\delta^2}{12}$. Then there exist a $w^*$-continuous mapping $\mu^\prime : S
\longrightarrow M(K)$, an open set $U$ in $S$, an open set $V$ of $K$ and $h_0 \in  C(K)$
satisfying the following   conditions:
\begin{itemize}
\item[i)]
$\vert \mu^\prime  (s)\vert (V) =0$ for every $s \in U$.
\item[ii)] $\int _{K} h_0 \ d \mu^\prime  (s) \ge  \Vert \mu ^\prime \Vert - \delta$  for every $s \in U$.
\item[iii)] $\Vert h_0 - f_0 \Vert < \delta $.
\item[iv)] $\Vert h_0  \Vert =1 $ and  $\vert h_0(t) \vert =1 \ \ \forall t \in K \setminus V$.
\item[v)] $\Vert \mu^\prime - \mu  \Vert < \delta  $.
%\end{enumerate}
\end{itemize}
\end{lemma}

\begin{proof}
Let us write $\mu_0:= \mu(s_0)$. By the Hahn decomposition theorem, there is a partition of $K$
into two measurable sets $K^+$ and  $K^-$ such that $K^+$ is a positive set for $\mu_0$ and $K^-$
is negative for $\mu_0$. For every $0 < x < 1$, consider two open subsets of $K$ given by
$$
O_x^+:=\{ t \in K \, :\, f_0(t) > x \}, \quad  O_x^-:= \{ t \in K \, :\, f_0(t) < -x \},
$$
and consider the set
$$
D_x:= \bigl(K^+ \cap O_x^+\bigr) \cup \bigl( K^- \cap O_x^-\bigr).
$$
Write $\alpha = \frac{\delta^2}{12}$. By the assumption, we have
\begin{align*}
   1- \alpha  & < \int_K f_0 \ d \mu_0 \le
  \vert \mu_{0} \vert  (D_x)  +x  \vert \mu_{0} \vert  ( K \setminus D_x) \\ & \le
 \vert \mu_{0}\vert (D_x)  +
  x  (1-  \vert \mu_{0}\vert (D_x) )   =
(1-x)  \vert \mu_{0}\vert (D_x) + x.
\end{align*}
Hence,
\begin{equation}\label{mu-0-Dx-grande}
\vert \mu_0 \vert (D_x) > 1- \frac{\alpha}{1-x} \, .
\end{equation}
Next, consider the open subset $W_x$ of $K$ given by
$$
W_x:= O_x^+ \cup O_x^-= \{ t \in K: \vert f_0(t) \vert > x \},
$$
and observe that, since $D_x \subset W_x$, we have
\begin{equation}
\label{mu-0-Wx-grande}
 \vert \mu_0 \vert (W_x) \ge  \vert \mu_0 \vert (D_x) \ge
 1- \frac{\alpha }{1-x} \ .
\end{equation}
Write $c:=1- \frac{\delta}{4}$ and choose real numbers $a$ and $b$ with $ 1-\delta < a < b< c<1$.
As the open subset $W_a$ contains $\overline{O_b^+ \cup O_b^-}$, there is $u \in C(K)$ such that
$0 \le u \le 1$, $u\equiv 1$ on $O_b^+ \cup O_b^-$ and $\supp u \subset W_a$. Since the support
of $u$ is contained in $W_a$ (where $f_0$ is separated from $0$), the function $h_0$
defined on $K$  by
$$
h_0(t)= \frac{f_0(t)}{\vert f_0 (t)\vert} u(t) + (1-u(t)) f_0(t)\,  \quad \text{ if} \quad
f_0(t)\ne 0,  \quad h_0(t) =0\, \text{ otherwise},
$$
is continuous and, actually, $h_0\in B_{C(K)}$. We claim that $\Vert h_0 - f_0 \Vert <  \delta $,
which guarantees condition iii). Indeed, if $t \in K \setminus W_a$, then $u(s)=0$ and so
$h_0(t)= f_0(t)$; if otherwise $t \in W_a$, we have that
$$
\vert h_0(t)- f_0(t) \vert = u(t) \left\vert \frac{f_0(t) }{\vert f_0(t) \vert} -
f_0 (t) \right\vert \le 1- \vert f_0(t) \vert < 1 - a < \delta,
$$
proving the claim. On the other hand, we know that $u\equiv 1$ in $O_b^+ \cup O_b^-$, so $\vert
h_0 \vert =1 $ on $\overline{O_b^+ \cup O_b^-}$. Therefore, if we write $V:=  K \setminus
\left[\overline{O_b^+ \cup O_b^-}\right]$, which is an open subset of $K$, the second part of
condition iv) is satisfied.

Next, as the open subsets $V $ and $W_c$ satisfy $ \overline{ V} \cap \overline{ W_c} =
\varnothing$, there is a function $f \in C(K)$ such that $0 \le f \le 1$, $f\equiv 1$ on $V$ and
$\supp f \subset K \setminus W_c$. Since $D_c \subset W_c$, we have that
\begin{align*}
  \int_{W_c}  h_0    \ d \mu_0 & =
\int_{D_c}  h_0    \ d \mu_0 + \int_{W_c \setminus D_c}  h_0    \ d \mu_0
 = \vert  \mu_{0} \vert (D_c) + \int_{W_c \setminus D_c}  h_0    \ d \mu_0
 \\
 & \ge \vert  \mu_{0} \vert (D_c) - \vert  \mu_{0} \vert (W_c \setminus D_c)
 \ge \vert  \mu_{0} \vert (D_c) - \vert  \mu_{0} \vert (K\setminus D_c)
 \ge 2\vert  \mu_{0} \vert (D_c) - \vert  \mu_{0} \vert (K).
\end{align*}
Therefore, by using \eqref{mu-0-Dx-grande}, we obtain that
\begin{equation}
\label{int-W-c-h0-grande}
  \int_{W_c}  h_0    \ d \mu_0  > 1 - 2\frac{ \alpha  }{ 1-c}
  = 1- \frac{2}{3} \delta.
\end{equation}
As a consequence,
\begin{align*}
\int_K h_0 (1-f) \ d \mu_0  = &   \int_{W_c}  h_0 (1-f)    \ d \mu_0 +
 \int_{K\setminus W_c}  h_0 (1-f)    \ d \mu_0
\\
 &  \\
 & \ge
 \int_{W_c}  h_0 (1-f)    \ d \mu_0 -
\vert \mu_{0}   \vert (K \setminus W_c)  \\
 &  \\
& > 1-  2\frac{ \alpha  }{ 1-c}  - \bigl( 1 - \vert \mu_{0}   \vert ( W_c)
\bigr)   \ \ \ \ \text{\rm (by \eqref{int-W-c-h0-grande})}
 \\
&  \\
 &  > 1- 3 \frac{ \alpha  }{ 1-c}  = 1 - \delta \ \ \ \
\ \ \ \ \text{\rm (by \eqref{mu-0-Wx-grande})} \\
\end{align*}
Now, in view of the $w^*$-continuity of $\mu$, the previous inequality, condition
\eqref{mu-0-Wx-grande} and Lemma~\ref{lem-w*lower-semicontinuous}, we get that
there exists an open neighborhood $U_0$ of $s_0$ such that
\begin{equation}
\label{int-h0-1-f-s-variacion-grande}
\int _K h_0(1-f) d \mu (s) > 1 - \delta \qquad \text{and} \qquad \vert \mu (s) \vert (W_c) > 1 -
\delta, \quad \forall s \in U_0.
\end{equation}
 We can also choose an open subset $U$ of $S$ such that $s_0 \in U$ and $\overline{U} \subset
U_0$, and a function $g\in C(S)$ such that $0 \le g \le 1 $, $g(U)=\{1\}$ and $\supp g \subset
U_0$. Define $\mu^\prime : S \longrightarrow M(K)$ by
$$
\mu^\prime (s) = \bigl( 1 - g(s) f \bigr) \mu (s) , \qquad  (s \in S),
$$
that is, $\mu^\prime (s)$ is the unique Borel measure  on $K$ satisfying
$$
\int _K \varphi \ d \mu^\prime (s) =  \int _K  \bigl( 1 - g(s) f \bigr) \varphi \
d \mu (s) \qquad \forall \varphi \in C(K).
$$
It is clear that $\mu^\prime $ is $w^*$-continuous. If $s \in U$, $g(s)=1$ and $f(V)=\{1\}$, so
condition i) is satisfied. Since $0\le f, g \le 1$, then $\Vert \mu ^ \prime \Vert \le \Vert \mu
\Vert=1$ and hence, in view of \eqref{int-h0-1-f-s-variacion-grande}, for every $s \in U$  we
have that
$$
\int _K h_0 \ d \mu^\prime (s) =  \int _K  \bigl( 1 - g(s) f \bigr) h_0 \ d \mu (s) >  1 - \delta
\ge \Vert \mu^\prime \Vert - \delta,
$$
so condition ii) is also satisfied.

We only have to check condition v), that is,  $\|\mu^\prime-\mu\|<\delta$. Indeed, if $s \in S
\setminus U_0$, then $g(s)=0$ and so $ \mu^\prime (s)= \mu (s)$; if, otherwise, $s \in U_0$, by
\eqref{int-h0-1-f-s-variacion-grande} and the fact that $f(W_c)=\{0\}$, we obtain that
$$
\Bigl \vert \int _K \varphi \ d (\mu (s) - \mu ^\prime (s) )  \Bigr \vert  = \Bigl
\vert \int _K \varphi  g(s) f \ d \mu (s)   \Bigr \vert \le
\vert \mu  (s) \vert (K \setminus W_c) \le  \Vert \mu \Vert  -  \vert \mu  (s)
\vert ( W_c) < \delta
$$
for every $\varphi \in B_{C(K)}$, proving the claim.

Finally, since $\Vert \mu \Vert  = 1$, condition v) implies that $\mu ^\prime\ne 0$ and, in view
of i), we deduce that $ K \ne  V $, so $K\setminus V $ is not empty and $\Vert h_0 \Vert =1$
since $\vert h_0(t) \vert = 1$ for every $t \in K \setminus V$.
\end{proof}

The last ingredient that we will use is the next iteration result due to Johnson
and Wolfe.

\begin{lemma}[\mbox{\cite[Lemma 2.4]{JoWo}}]\label{Jo-Wolf-Lemma24}
Let $\mu:S\longrightarrow M(K)$ be $w^*$-continuous and $\delta>0$. Suppose there is an open set
$U\subset S$, an open set $V\subset K$, $s_0\in U$ and $h_0\in C(K)$ with $\|h_0\|=1$ such that
\begin{enumerate}
\item[i)] if $s\in U$, then $|\mu(s)|(V)=0$,
\item[ii)] $\int_K h_0\,d\mu(s_0) \geq \|\mu\|-\delta$,
\item[iii)] $|h_0(t)|=1$ for $t\in K\setminus V$.
\end{enumerate}
Then, for any $\frac23<r<1$ there exist a $w^*$-continuous function $\mu^\prime:S\longrightarrow
M(K)$ and a point $s_1\in U$ such that
\begin{enumerate}
\item[i)] if $s\in U$, then $|\mu^\prime(s)|(V)=0$,
\item[ii)] $ \int_K h_0\,d\mu^\prime(s_1)  \ge   \|\mu^\prime\|-r\delta $,
\item[iii)] $\|\mu^\prime - \mu\|\leq r\delta$.
\end{enumerate}
\end{lemma}

The next result improves \cite[Theorem 1]{JoWo}.

\begin{theorem}
\label{BPB-CK-CS}
Let $K$ and $S$ be compact Hausdorff topological spaces. Then the pair $(C(K),
C(S))$ has the Bishop-Phelps-Bollob\'{a}s property for operators in the real case.
Moreover, the
\linebreak[4]
function $\eta$ satisfying Definition~\ref{def-BPBP-operadores} does not depend on
the spaces $K$ and $S$ (in fact one can take $\eta (\varepsilon)=
\frac{\varepsilon^2 }{12\cdot 6^2} $).
\end{theorem}

\begin{proof}
Let us fix $\frac{2}{3} < r < 1 $. Given $ 0< \varepsilon < 2$ let us choose $0 <
\delta < \varepsilon \frac{1-r}{2}$. Assume that $T_0 \in S_{\mathcal{L} (C(K),
C(S))} $ and $f_0 \in S_{C(K)}$ satisfy that $\Vert T_0 (f_0) \Vert > 1 -
\frac{\delta ^2 }{12}$. Then, there is an element $s_1 \in S$ such that $\left\vert
\bigl[T_0 (f_0)\bigr] (s_1) \right\vert > 1 - \frac{\delta ^2 }{12}$. By using
$-f_0$ instead of $f_0$, if necessary, we may assume that $T_0 (f_0) (s_1) > 1 -
\frac{\delta ^2 }{12}$. Therefore, we can apply Lemma~\ref{funct-attains-close} to
the $w^*$-continuous function $\mu_0: S \longrightarrow M(K)$ associated with the
operator $T_0$ (i.e.\ $\mu_0(s)=T_0^*(\delta_s)$ for every $s\in S$) to get that
there exist a function $h_0 \in S_{C(K)}$, an open set $U$ in $S$, an open set $V$
of $K$ and a $w^*$-continuous function $\mu_1:S \longrightarrow M(K)$ satisfying
the following conditions:
\begin{itemize}
\item[i)]
$\vert \mu_1  (s)\vert (V) =0$ for every $s \in U$.
\item[ii)] $\int _{K} h_0\, d\mu_1  (s) \ge  \Vert \mu_1\Vert - \delta$  for every $s \in U$.
\item[iii)] $\Vert h_0 - f_0 \Vert < \delta $.
\item[iv)]   $ \Vert h_0 \Vert = 1 $ and  $\vert h_0(t) \vert =1 \ \ \forall t \in K \setminus V$.
\item[v)] $\Vert \mu_1 - \mu_0  \Vert < \delta  $.
\end{itemize}
Now, by using Lemma~\ref{Jo-Wolf-Lemma24}, we  inductively construct a sequence
$\{\mu_n\}$ of $w^*$-continuous functions from $S$ into $M(K)$ and a sequence $\{
s_n\}$ in $U$ satisfying
\begin{itemize}
\item[i)] $\Vert \mu_{n+1}- \mu_{n}  \Vert \le r^n \delta  $.
\item[ii)]
$\Vert \mu_n \Vert \le \int _K h_0 \,d\mu_n(s_n) + r^n \delta $
\item[iii)] $\vert \mu_n   (s)\vert (V) =0$ for every $s \in U$ and $n \in \N$.
\end{itemize}
If for every $n \in \N$, we write $T_n \in \mathcal{L}(C(K), C(S))$ to denote the
bounded linear operator associated with the function $\mu_n$, we
may rewrite i) and ii) as
\begin{equation}
\label{T-n-h0-close-norm}
\Vert T_{n+1}- T_{n}  \Vert \le r^n \delta \qquad \text{and} \qquad \Vert T_n \Vert \le \Vert T_n
(h_0) \Vert  + r^n \delta.
\end{equation}
Since $0 < r < 1$, the  previous condition implies that $\{T_n\} $ is a Cauchy sequence, so it
converges to an operator $T \in \mathcal{L}(C(K), C(S))$ satisfying
$$
\Vert T - T_0 \Vert \le \sum _{k=0}^\infty \Vert T_{k+1} -  T_k  \Vert \le  \sum _{k=0}^\infty
r^k \delta = \delta \frac{1}{1-r} < \frac{\varepsilon}{2}.
$$
By taking limit in the right-hand side of
\eqref{T-n-h0-close-norm}, we also have that
$$
\Vert T \Vert \le \Vert T (h_0 ) \Vert
$$
and, since $h_0 \in S_{C(K)}$, $T$ attains its norm at $h_0$.

Finally, we have that
$$
\left\vert  1 - \Vert T \Vert \right\vert \le \bigl\vert  \Vert T_0 \Vert - \Vert T \Vert
\bigr\vert \le \Vert T_0 - T \Vert <\frac{\varepsilon}{2} < 1,
$$
so $T\neq 0$, $\frac{T}{\|T\|}$ also attains its norm at $h_0$ and
$$
\left \Vert  \frac{T}{\Vert T \Vert } - T_0 \right \Vert \le \left \Vert  \frac{T}{\Vert T \Vert
} - T \right \Vert + \Vert  T - T_0  \Vert = \bigl\vert 1 - \Vert T \Vert \bigr\vert + \Vert T -
T_0 \Vert < \varepsilon.
$$
As we already knew that $\Vert h_0 -f_0 \Vert < \delta < \varepsilon$, this shows that the pair
$(C(K), C(S))$ satisfies the Bishop-Phelps-Bollob\'{a}s Theorem for operators with $\eta
=\frac{\delta ^2}{12}$.
\end{proof}

\section{Compact operators from a space of continuous functions into a uniformly convex
space}
\label{sec:c(K)toUniformlyconvex}

Our purpose now is to prove the Bishop-Phelps-Bollob\'{a}s property for compact
operators. The following result due to Kim will play an essential role:

\begin{lemma}[\mbox{\cite[Theorem 2.5]{Kim}}]
\label{teo-Kim-UC}
Let $Y$ be a uniformly convex  space. For every $0 < \varepsilon < 1$,  there is $
0 < \gamma (\varepsilon) <1$ with the following property:
\newline
given  $n \in \N$, $T \in S_{\mathcal{L} (\ell_\infty ^n , Y)} $ and $ x_0 \in S_{\ell_\infty
^n}$ such that $ \Vert T x_0 \Vert > 1 - \gamma (\varepsilon )$, there exist $S \in
S_{\mathcal{L} (\ell_\infty ^n , Y)}$ and $ x_1 \in S_{\ell_\infty ^n}$ satisfying
$$
\Vert S x_1 \Vert =1, \quad \Vert S - T \Vert < \varepsilon \quad \hbox{and} \quad \Vert x_1 -
x_0 \Vert < \varepsilon \ .
$$
\end{lemma}

It is easy to show that the function $\gamma $ in the previous result satisfies
$\lim _{t
\to 0+} \gamma (t)=0$.

Let $L$ be a locally compact Hausdorff topological space. As usual, $C_0(L)$ will
be the space either of real or  complex continuous functions on $L$ with limit zero
at infinity. We recall that $C_0(L)^*$ can be identified with the space $M(L)$ of
regular Borel measures on $L$ by the Riesz representation theorem.

For every $f \in C_0(L)$ and every (non-empty) set  $S \subset L$, we define $\osc (f, S)$  by
$$
\osc (f, S) =  \sup_{x,y \in S}\vert f(x)- f(y)\vert
$$

  The next result generalizes Lemma 3.1 and
Proposition 3.2 of \cite{JoWo} to $C_0(L)$. Its proof is actually based on the
proof of these results.

\begin{proposition}
\label{jowopro}
Let $L$ be a locally compact Hausdorff topological space and  let $Y$  be  a Banach space. For
every $\varepsilon>0$, $T\in \mathcal{K}(C_0(L),Y)$  and $f_0\in C_0(L)$, there exist a positive
regular Borel measure $\mu$, a non-negative integer $m$, pairwise disjoint compact subsets $K_j$
of $L$ and $\varphi_j\in C_0(L)$ for $1 \le j \le m$,  satisfying the following conditions:
\begin{enumerate}
\item $\osc (f_0, K_j) < \varepsilon$.
\item  $0 \le \varphi _j \le 1$ and $\varphi _{j}\equiv 1$ on $K_j$.
\item $\supp \varphi_i \cap \supp \varphi_j = \emptyset$ for $i \ne j$.
\item  The operator $P:  C_0(L) \longrightarrow C_0(L)$ given by
$$
%\displaystyle
P(f):= \sum _{j=1}^m \frac{1}{\mu (K_j )} \Bigl( \int_{K_j} f \ d \mu \Bigr)
\varphi  _j, \quad \forall f \in C_0(L),
$$
 is a norm-one projection from $C_0(L)$ onto the linear
span of $\{\varphi_1,\ldots,\varphi_m\}$ that also satisfies $\|T-TP\|<\varepsilon $.
\end{enumerate}
\end{proposition}

\begin{proof}
Since $T$ is a compact operator,  the adjoint operator $T^*$ is a compact operator
from $Y^*$ into $C_0(L) ^*$, so we may take a finite $\frac{\varepsilon}{4}$-net
$\{\mu_1,\ldots,\mu_t\}$ of $T^*(B_{Y^*})\subset C_0(L)^*\equiv M(L)$. We define
the (finite regular) measure $\mu$ by $\mu=\sum^t_{i=1}|\mu_i |$.  For each $1\leq
i\leq t$, we have that $\mu_i\ll \mu$, hence the Radon-Nikod\'{y}m theorem allows
us to find a function $g_i\in L_1(\mu)$ such that $\mu_i=g_i\mu$. Since the set of
simple functions is dense in $L_1(\mu)$, we may choose a set of simple functions
$\{s_i\,:\, i=1,\ldots,  t\}$ such that $\|g_i-s_i\|_1<\frac{\varepsilon}{12}$ for
every $1\le i \le t $. Next, consider a finite family $(A_j)_{j=1}^{m_0}$ of
pairwise disjoint measurable sets such that for every $1 \le i \le t$, there is a
family $(\alpha_j^i)_{j=1}^{m_0}$ of scalars such that
$s_i=\sum\limits^{m_0}_{j=1}\alpha_j^i\chi_{A_j}$. Let $M$ be a positive real
number satisfying $ M \ge \max \bigl\{ |\alpha_j^i|: 1\le i \le t, 1 \le j \le
m_0\bigr\} $.  Since $\mu$ is regular, for each $1 \le j \le m_0$ we find a compact
set $C_j \subset A_j$ such that $\mu(A_j\setminus
C_j)<\frac{\varepsilon}{12{m_0}M}$. As $f_0$ is continuous and each $C_j$ is
compact, we may divide each $C_j$ into a family of Borel sets $(B_j^p)_{p=1}^{n_j}$
such that
$$
\osc (f_0, B_j^p)<\varepsilon \qquad  \forall 1 \le j \le m_0, 1 \le p \le n_j .
$$
Applying the regularity again, for each $j$ and $p$, there is  a compact set $K_j^p
\subset B_{j}^p $ such that $\mu(B_j^p\setminus K_j^p)<\frac{\varepsilon}{12
m_0n_jM}$. Finally, choose suitable $m\in \mathbb{N}$, a rearrangement
$(K_j)_{j=1}^m$ of the family $\{K_j^p\,: \, 1 \le j \le m_0,\, 1 \le p \le n_j,\;
\mu \bigl( K_j ^p  \bigr) > 0\}$ and scalars $(\beta_j^i)$ for $ j \le m $ and $i
\le t$ such that
$$
\sum^m_{j=1}\beta_j^i \chi_{K_j}=\sum^{m_0}_{j=1} \alpha_j^i \biggl(  \sum_{p=1}^{n_j} \chi_{K_j^p}
\biggr) \ .
$$
Using Urysohn lemma, we may choose a family $(\varphi_j)_{j=1}^m$
in $C_0(L)$ satisfying that $0\leq \varphi_j\leq 1$, \;
$\varphi_j\equiv 1$ on $K_j$ for each $j \le m$ and $\supp \varphi_i \cap \supp
\varphi_j = \emptyset$ for every $i \neq j$.

To finish the proof, we only have to check (4). Indeed, for $i=1,\ldots,t$, we
write $\nu_i = \sum_{j=1}^m\beta_j^i\chi_{K_j}\mu\in M(L)=C_0(L)^*$. By defining
the operator $P$ as in condition (4), it is easy to check that $P$ is a norm one
projection onto the linear span on $\{ \varphi_1, \ldots, \varphi _m\}$ and
$P^*\nu_i=\nu_i $ for each $1\leq i \le t$.  Therefore,  for every $ 1 \le i \le t$
we have that
\begin{align*}
\|\mu_i-P^*\nu_i\|
&= \|g_i\mu-\nu_i\|\leq\|g_i\mu-s_i\mu\|+\|s_i\mu-\nu_i\|\\
&\leq \|g_i-s_i\|_1+ \biggl \Vert s_i\mu-\sum^{m_0}_{j=1}\alpha_j^i\chi_{C_j}\mu \biggr \Vert
 +\biggl \Vert \sum^{m_0}_{j=1}\alpha_j^i\chi_{C_j}\mu- \sum^m_{j=1}\beta_j^i\chi_{K_j}\mu \biggr \Vert \\
&<\frac{\varepsilon}{12} +\sum^{m_0}_{j=1} \bigl \vert \alpha_j^i \bigr \vert \mu
\bigl(A_j\setminus C_j
\bigr) +
\biggl
\Vert
\sum^{m_0}_{j=1} \biggl( \alpha_j^i \sum_{p=1}^{n_j} \chi_{B_j^p}  \biggr) \mu -
 \sum^{m_0}_{j=1} \biggl( \alpha_j^i \sum_{p=1}^{n_j} \chi_{K_j^p} \biggr)  \mu \biggr \Vert \\
&<\frac{\varepsilon}{12}+ \frac{\varepsilon}{12} +
\sum^{m_0}_{j=1} \bigl \vert \alpha_j^i  \bigr \vert \sum_{p=1}^{n_j} \mu \bigl( B_{j}^p \setminus K_{j}^p \bigr)\\
&<\frac{\varepsilon}{12} + \frac{\varepsilon}{12} + \frac{\varepsilon}{12} =
 \frac{\varepsilon}{4}\ .
\end{align*}
Since $\{ \mu_1, \ldots, \mu _t \}$ is a  $\frac{\varepsilon}{4}$-net of
$T^*(B_{Y^*})$, the above inequality  shows that $\{ \nu_1, \ldots, \nu _t \}$ is a
$\frac{\varepsilon}{2}$-net of $T^*(B_{Y^*})$. Now, given $y^*\in B_{Y^*}$, we can
choose $i \le t$ satisfying $\|\nu_i-T^*y^*\|< \frac{\varepsilon}{2}$ and observe
that
\begin{align*}
\|T^*y^*-P^*T^*y^*\|
&\leq \|T^*y^*-\nu_i\|+\|\nu_i-P^*T^*y^*\|\\
&= \|T^*y^*-\nu_i\|+\|P^*\nu_i-P^*T^*y^*\|  \leq 2 \|T^*y^*-\nu_i\|< \varepsilon.
\end{align*}
Hence, we have $\|T-TP\|=\|T^*-P^*T^*\|<\varepsilon$, as desired.
\end{proof}

The following result shows that $\mathcal{K}(C_0(L),Y)$ satisfies the Bishop-Phelps-Bollob\'{a}s
property for every locally compact Hausdorff topological space $L$ and every uniformly convex
space $Y$, and that the function $\eta(\eps)$ involved in the definition of the property does not
depend on $L$.

\begin{theorem}
\label{BPB-comp-C(K)-UC}
Let $Y$ be a uniformly convex Banach space. For  every $0 < \varepsilon < 1$ there is $ 0 < \eta
(\varepsilon ) < 1 $ such that for any locally compact Hausdorff topological space $L$, if $T\in
S_{\mathcal{K}(C_0(L), Y)}$ and $f_0 \in S_{C_0(L)} $ satisfy $\Vert T f_0 \Vert > 1 -
\eta(\varepsilon )$, there exist $S \in S_{\mathcal{K}(C_0(L), Y)}$ and $g_0 \in S_{ C_0(L)}$
such that
$$
\Vert S g_0 \Vert =1, \quad \Vert S - T \Vert < \varepsilon \quad \hbox{and} \quad \Vert g_0-
f_0  \Vert < \varepsilon \ .
$$
\end{theorem}

\begin{proof}
Given $0 < \varepsilon < 1$, we choose $0 < \delta < \frac{\varepsilon}{4}$ such that  $0 <
\gamma (\delta ) < \frac{\varepsilon}{4}$, where $\gamma (\delta)$ satisfies the statement of
Lemma~\ref{teo-Kim-UC}. We also consider $\alpha$ such that $0 <   \alpha <  \min\bigl\{ \delta,
\frac{\gamma (\delta)}{2}\bigr\}$ and $\eta(\eps):=\alpha>0$.

Fix $0<\eps<1$, $T\in S_{\mathcal{K}(C_0(L), Y)}$ and $f_0 \in S_{C_0(L)}$ with
$\|Tf_0\|>1-\eta(\eps)=1-\alpha$. Applying Proposition~\ref{jowopro}, we get a positive regular
Borel measure $\mu$ on $L$, a non-negative integer $m$, pairwise disjoint compact subsets $K_j$
of $L$ and $\varphi _j\in C_0(L)$ ($1 \le j \le m$) such that
\begin{enumerate}
\item $\osc (f_0, K_j) < \alpha$,
\item For every $ 1 \le j \le m$, $ 0 \le \varphi _j \le 1$ and $\varphi _{j}\equiv 1$ on $K_j$,
\item $\supp \varphi_i \cap \supp \varphi_j = \emptyset$ for $i \ne j$,
\item $\|T-TP\|<\alpha$,
\end{enumerate}
where $P\in \mathcal{L}(C_0(L))$ is given by
$$
P(f):= \sum _{j=1}^m \frac{1}{\mu (K_j )}  \Bigl( \int_{K_j} f \ d \mu \Bigr) \varphi_j \qquad
\bigl(f\in C_0(L)\bigr),
$$
and it is a norm-one projection onto the linear span of $\{\varphi_1, \ldots,\varphi_m\}$.

Now, if $t \in K_j$ for some $j=1,\ldots, m$, we obtain that
\begin{align*}
\bigl\vert [P(f_0)] (t) - f_0(t) \bigr\vert &= \left\vert \frac{1}{\mu (K_j) } \int_{K_j}  \bigl(
f_0 (s) - f_0(t)\bigr)  \ d \mu(s)  \right\vert
\\
&\leq \frac{1}{\mu (K_j) }\int_{K_j} \bigl\vert f_0 (s) - f_0(t) \bigr\vert \ d \mu (s)  \leq
\alpha.
\end{align*}
Hence
\begin{equation}\label{P-f0-f0}
 \max \left\{  \bigl\vert [P(f_0) - f_0](t) \bigr\vert \,:\, t \in  \bigcup\nolimits_{j=1}^m K_j
 \right\}\leq \alpha.
\end{equation}
We also have that
\begin{equation}\label{TP-f0-grande}
\Vert TP(f_0) \Vert \ge \Vert T(f_0) \Vert - \Vert T - T P \Vert > 1 - 2\alpha  >
1 - \gamma (\delta ),
\end{equation}
and this implies that
$$
1 - 2 \alpha \le \Vert TP \Vert \le 1 \quad \text{and} \quad  1 - 2 \alpha \le  \Vert
P (f_0) \Vert.
$$
Since the functions $\{ \varphi _j \,:\, 1 \le j \le m\}$ have pairwise disjoint
support, the linear operator $\Phi: \mathrm{lin}\{\varphi_1, \ldots,   \varphi_m\}
\longrightarrow \ell_{\infty}^m$ satisfying $\Phi(\varphi_j) =
e_j$ for every $j=1,\ldots, m$ is an onto linear isometry (where
$\{e_1,\ldots,e_m\}$ is the natural basis of $\ell_\infty^m$). We define $U_1 := T
\circ \Phi ^{-1}:\ell _{\infty}^m \longrightarrow Y$ and observe that, clearly,
$\|U_1\|\leq 1$. On the other hand, the element $x^0:= \Phi (P(f_0))\in
B_{\ell_\infty^m}$ satisfies that $U_1 (x^0)=  TP (f_0)$ so, in view of
\eqref{TP-f0-grande},
$$
\left\Vert U_1 (x^0 ) \right\Vert > 1 - \gamma
(\delta ) \quad \text{and so} \quad \Vert U_1 \Vert, \Vert x^0 \Vert > 1 -
\gamma (\delta ) > 0.
$$
We consider $U= \frac{U_1}{\Vert U_1 \Vert }$ and apply Lemma~\ref{teo-Kim-UC} to the pair
$\Bigl( U, \frac{x^0}{\Vert x^0 \Vert } \Bigr)$, to get an operator $V: \ell_{\infty }^m
\longrightarrow Y$ with $\Vert V \Vert =1$ and $x^1 \in S_{ \ell_{\infty}^m}$ with
\begin{equation}\label{V-menos-U}
\Vert V - U \Vert < \delta, \quad \left\Vert x^1 - \frac{x^0}{\Vert x^0 \Vert}\right \Vert <
\delta, \quad \text{and} \quad \Vert V \bigl( x^1 \bigr)  \Vert =1 \ .
\end{equation}
We clearly have that
\begin{equation}
\label{U-menos-U1}
\Vert U - U_1 \Vert = \left\Vert \frac{U_1} {\Vert U_1 \Vert } - U_1 \right\Vert= \bigl|1 - \Vert
U_1 \Vert\bigr| \leq \gamma (\delta )
\end{equation}
and, also,
\begin{equation}\label{x1-menos-x0}
\left\Vert x^1 - x^0  \right\Vert \leq  \left  \Vert  x^1 - \frac{x^0} {\Vert x^0 \Vert }
\right\Vert + \left\Vert  \frac{x^0} {\Vert x^0 \Vert }  - x^0 \right \Vert \leq
\delta + 1 - \Vert x^0 \Vert \le \delta + \gamma (\delta ).
\end{equation}
As a consequence,
$$
\Vert V - U_1 \Vert \le  \Vert V - U \Vert  + \Vert U - U_1 \Vert   \le \delta +
\gamma (\delta ).
$$
Finally, we define the operator $S: C_0(L) \longrightarrow Y$ given by $ S(f):= V
(\Phi (P(f)))$ for every $f \in C_0(L)$, which is clearly a compact operator and
satisfies $\|S\|\leq 1$. Consider the element $f_1 = \sum _{j=1}
^m x^1 (j)
\varphi _j\in B_{C_0(L)}$. It is clear that $P(f_1)= f_1$, $\Phi (f_1)= x^1$ and
that
$$
\Vert S \Vert \ge \Vert S ( f_1) \Vert = \Vert V \Phi P (f_1 ) \Vert =
 \bigl \Vert V \bigl(x^1 \bigr) \bigr\Vert =1\ .
$$
We deduce that $f_1\in S_{C_0(L)}$ and that  $\Vert S
\Vert = \Vert S (f_1 ) \Vert =1$, i.e.\ $S$ attains its norm on $f_1$.

Next, we estimate the distance between $S$ and $T$ as follows
\begin{align*}
\Vert S - T \Vert &= \Vert V \Phi P - T \Vert \le \Vert V \Phi P - TP  \Vert +  \Vert TP - T \Vert \\
& \le \Vert V \Phi P -  U \Phi P   \Vert +   \Vert U \Phi P - TP  \Vert + \alpha\\
& \le\Vert V - U \Vert +  \Vert U \Phi P - T \Phi ^{-1} \Phi P  \Vert  + \alpha \\
& <\delta  + \Vert U - T \Phi ^{-1} \Vert + \alpha  \qquad \text{(by \eqref{V-menos-U})}\\
& \le 2\delta  + \Vert U - U_1 \Vert \le   2 \delta + \gamma (\delta )<
\eps \qquad \text{(by \eqref{U-menos-U1})}.
\end{align*}
On the other hand,
\begin{align*}
\max  \left\{  \bigl\vert [f_1 - f_0] (t) \vert \,:\, t \in \bigcup\nolimits_{ j=1}^m K_j  \right\}
&=\max _{ 1 \le j \le m}   \max_{t
\in K_j} \, \bigl\vert x^1(j) - f_0 (t) \bigr\vert \\
&\le\max _{ 1 \le j \le m}     \left\{ \vert x^1(j) - x^0(j) \vert   +  \max_{t \in K_j}
 \bigl\vert x^0(j) - f_0 (t) \bigr\vert \right\} \\
&\le\Vert x^1 - x^0 \Vert \,  + \, \max _{ 1 \le j \le m}  \max_{t \in K_j} \bigl\vert
[P(f_0)
- f_0] (t) \bigr\vert\\
&\le \delta + \gamma (\delta ) + \alpha
 <  2 \delta + \gamma (\delta )
\qquad \text{(by \eqref{x1-menos-x0} and \eqref{P-f0-f0})}.
\end{align*}
Hence, there exists an open set $O \subset L$ such that
\begin{equation}\label{meseta}
\bigcup\nolimits_{j=1} ^m K_j \subset O, \quad  \bigl\vert [f_1-f_0](t) \bigr\vert   < 3 \delta +
\gamma (\delta ) \qquad \bigl(t \in O \bigr).
\end{equation}
By Urysohn Lemma again, there is $\psi \in C_0(L)$ such that $ 0 \le \psi \le 1 $,
$\psi \equiv 1$ on $\bigcup_{k=1}^m K_j$ and $\supp \psi \subset O$. We write
$g_0:= \psi f_1 + (1-\psi) f_0\in B_{C_0(L)}$ and we claim that $S$ attains its
norm at $g_0$ and that $\|f-g_0\|<\eps$, which finishes the proof. Indeed, on  one
hand, it is clear that the restriction of $g_0$ to $\bigcup_{k=1}^m K_j$ coincides
with $f_1$. It follows that $P(g_0)= P(f_1)$ and so $S(g_0)=S(f_1)$ and $S$ attains
its norm at $g_0$. On the other hand, for $t \in L\setminus O$ we have
$g(t)=f_0(t)$. If, otherwise, $t\in O$, condition \eqref{meseta} gives that
\begin{equation*}
\bigl\vert g_0 (t) - f_0 (t) \bigr\vert = \left\vert \psi(t) \bigl( f_1(t)  - f_0(t)\bigr)  \right \vert <
3 \delta + \gamma (\delta )<\eps.\qedhere
\end{equation*}
\end{proof}

\section{Compact operators into a predual of an $L_1(\mu)$-space}
\label{sec:intoPredual_of_L_1}
Our goal is to show that the space of compact operators from an arbitrary Banach
space into an isometric predual of an $L_1$-space has the
Bishop-Phelps-Bollob\'{a}s property in both the real and the complex case.

We need a preliminary result which follows easily from the Bishop-Phelps-Bollob\'{a}s theorem. It
is also a very particular case of \cite[Theorem~2.2]{AAGM}.

\begin{lemma}\label{lemma-ellinfty-m-B+}
For every $0<\eps<1$, there is $0<\eta'(\eps)<1$ such that for every positive
integer $n$ and every Banach space $X$, the pair $(X,\ell_\infty^n)$ has the BPBp
for operators with this function $\eta'(\eps)$. More concretely, given an operator
$U\in S_{\mathcal{L}(X,\ell_\infty^n)}$ and an element  $x_0\in S_X$ such that
$\|U(x_0)\|>1-\eta'(\eps)$, there exist $V\in S_{\mathcal{L}(X,\ell_{n}^\infty)}$
and $z_0 \in S_X$ satisfying
$$
\Vert V z_0 \Vert =1 , \qquad \Vert z_0 - x_0 \Vert  < \varepsilon \quad \text{and} \quad \Vert V
- U \Vert  < \varepsilon .
$$
\end{lemma}

\begin{theorem}
\label{BPB-compact-preduals-L1}
For every $ 0< \varepsilon < 1$ there is $\eta(\eps) > 0$ such that if $X$ is any Banach space,
$Y$ is a predual of an $L_1$-space, $T \in S_{\mathcal{K}(X,Y)}$ and $x_0 \in S_{X}$ satisfy
$\Vert T x_0 \Vert > 1 - \eta(\eps)$, then there exist $S\in S_{\mathcal{F}(X,Y)}$ and $z_0 \in
S_X$ with
$$
\Vert S z_0 \Vert =1 , \qquad \Vert z_0 - x_0 \Vert  < \varepsilon \quad \text{and}
\quad \Vert S - T \Vert  < \varepsilon.
$$
\end{theorem}
\begin{proof}
For any $ 0< \varepsilon < 1$  we take $\eta( \varepsilon)= \min\{
\frac{\varepsilon}{4}, \eta'(\varepsilon/2)\}$, where $\eta'$ is the function
provided by the previous lemma.

Fix $0<\varepsilon <1$, $T\in S_{\mathcal{K}(X,Y)}$ and $x_0 \in S_{X}$ satisfying
$\Vert T x_0 \Vert > 1 - \eta( \varepsilon)$. Let us choose  a positive number
$\delta$ with $\delta  <  \frac{1}{4} \min \bigl \{ \frac{\varepsilon}{4}, \Vert T
(x_0) \Vert -1 + \eta  ^\prime \bigl( \frac{\varepsilon}{2} \bigr) \bigr\}$ and let
$\{ y_1, \ldots, y_n\}$ be a $\delta$-net of $T(B_X)$.  In view of \cite[Theorem
3.1]{LaLi} and \cite[Theorem 1.3]{NieOls}, there is a subspace $E\subset Y$
isometric to $\ell_\infty^m$ for some natural number $m$ and such that $\dist (y_i,
E) < \delta$ for every $i \le n$. Let  $P: Y \longrightarrow Y$ be a norm one
projection onto $E$. We will check that $\Vert PT- T \Vert < 4\delta$.  In order to
show that we fix any element $x \in B_X$ and so $\Vert Tx - y_i \Vert < \delta$ for
some $i \le n$. Let $e \in E$ be any element satisfying $\Vert e-y_i \Vert <
\delta$. Then we have
$$
\Vert  T(x)-PT(x)\Vert \le  \Vert  T(x)-y_i\Vert + \Vert  y_i - e\Vert +\Vert e-
PT(x)\Vert
$$
$$
\le  2 \delta + \Vert P(e) -PT(x)\Vert  \le 2 \delta + \Vert e-T(x) \Vert
$$
$$
\le  2 \delta + \Vert e  -y_i \Vert + \Vert y_i -T(x)\Vert  < 4 \delta.
$$
So $\Vert PT \Vert > \Vert T \Vert - 4 \delta = 1 -4 \delta > 0$. As a consequence
we also obtain that
$$
\Vert PT(x_0) \Vert > \Vert T(x_0) \Vert - 4 \delta  > 1 - \eta^\prime \bigl(
\frac{\varepsilon}{2} \bigr)\ .
$$
Hence the operator $R= \frac{PT}{\Vert PT \Vert }$ satisfies $ \Vert R(x_0) \Vert >
1 - \eta^\prime  \bigl( \frac{\varepsilon}{2} \bigr) $. Since $E$ is isometric to
$\ell_{\infty}^m$,  by Lemma~\ref{lemma-ellinfty-m-B+} there exist an operator $S
\in \mathcal{L}(X,E)\subset \mathcal{L}(X,Y)$ with $\|S\|=1$ and $z_0 \in S_X$
satisfying that
$$
\Vert S -  R \Vert  < \frac{\varepsilon}{2}, \qquad \Vert z_0 - x_0 \Vert <
\frac{\varepsilon}{2}, \quad \text{and} \quad \Vert S z_0 \Vert =1.
$$
Finally, we have that
\begin{align*}
\Vert S - T \Vert &\le \Vert S -  R \Vert + \Vert  R - PT
 \Vert  + \Vert PT - T \Vert  \\
&< \frac{\varepsilon}{2} +  1 - \Vert PT \Vert   + 4\delta  \\
& < \frac{\varepsilon}{2} +   8 \delta <  \varepsilon .\qedhere
\end{align*}

\end{proof}

%\vspace*{3cm}

\bibliographystyle{amsalpha}

\end{document}